\documentclass[11pt, a4paper, twoside,reqno]{amsart}
\usepackage{amsmath,amsthm,amsfonts,amssymb}
\usepackage[utf8]{inputenc}
\usepackage{mathrsfs}
\usepackage{graphics,graphicx}
\usepackage{mathabx}
\usepackage[dvipsnames]{xcolor}
\usepackage{enumerate}
\usepackage{icomma}
\usepackage{empheq}
\usepackage{pdflscape}
\usepackage{epstopdf}
\usepackage{scalerel}
\usepackage{rotating}
\usepackage[shortlabels]{enumitem}
\usepackage[bookmarks=true,breaklinks=true,bookmarksnumbered = true,colorlinks=false,hidelinks]{hyperref}
\usepackage{tikz}

\newcommand{\vstack}{
  \mathbin{
    \tikz[baseline=-0.582ex]
      \draw[line width=.1ex, line cap=round]
        (0,0) circle (.75ex)
        (0,-.5ex) -- (0,.5ex);
  }
}

\newcommand{\subvstack}{
  \mathbin{
    \tikz[baseline=-0.449ex]
      \draw[line width=.075ex,line cap=round]
        (0,0) circle (.5625ex)
        (0,-.375ex)--(0,.375ex);
  }
}

\newcommand{\hstack}{
  \mathbin{
    \tikz[baseline=-0.582ex]
      \draw[line width=.1ex, line cap=round]
        (0,0) circle (.75ex)
        (-.5ex,0) -- (.5ex,0);
  }
}
\newcommand{\subhstack}{
  \mathbin{
    \tikz[baseline=-0.449ex]
      \draw[line width=.075ex,line cap=round]
        (0,0) circle (.5625ex)
        (-.375ex,0)--(.375ex,0);
  }
}
\definecolor{arrow}{RGB}{154, 32, 64}

\theoremstyle{plain}
\newtheorem*{thm0}{Theorem}
\newtheorem*{ques}{Question}

\newtheorem{thm}{Theorem}[section]

\newtheorem{lem}[thm]{Lemma}  
\newtheorem{prop}[thm]{Proposition} 
\newtheorem{cor}[thm]{Corollary} 

\theoremstyle{definition}
\newtheorem{defn}[thm]{Definition} 
\newtheorem{ex}[thm]{Example}

\theoremstyle{remark}
\newtheorem{rmq}[thm]{Remark}

\newtheorem{conv}[thm]{Convention}

\newsavebox{\mybox}
\usepackage{graphicx} 

\begin{document}
\title{A welded extension of \textit{4-}equivalence}
\author{Emmanuel Graff}

\begin{abstract}
We introduce a local move defining a welded extension of \textit{4-}equivalence. It induces the same quandle relations as \textit{4-}equivalence, giving a common algebraic framework for both the classical and welded settings. We also develop a version of arrow calculus adapted to this welded extension and apply it to the study of virtual tangles. This leads to a finite classification of algebraic virtual tangles without closed components under the resulting equivalence relation. As an application, we prove that the closure of any such tangle can be reduced to the unknot, giving a partial positive answer to a welded analogue of Nakanishi's conjecture.
\end{abstract}

\maketitle

\section*{Introduction}

A \emph{4-move} is a local operation in knot theory replacing two consecutive crossings with their opposite, as illustrated in Figure~\ref{fig4move}.

\begin{figure}[!htbp]
\centering
\includegraphics{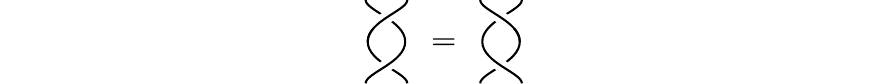}
\caption{A \textit{4-}move.}
\label{fig4move}
\end{figure}

The equivalence relation generated by \textit{4-}moves and ambient isotopies is called \emph{\textit{4-}equivalence}.

In 1979, Nakanishi conjectured that every knot is \textit{4-}equivalent to the trivial knot. This is no longer true for links, the Hopf link being a classical counterexample. In 1985, Kawauchi proposed a link analogue of the conjecture, asserting that any two link-homotopic links are \textit{4-}equivalent.

Both conjectures appear in~\cite{KirbyProbLowDimTop}. While Kawauchi's conjecture is now known to be false~\cite{DabkowskiPrzytyckiUnexpConnBetweenBurnsideGrpsKnotTheory}, Nakanishi's original conjecture remains open. It has been verified for the following families of knots:
\begin{itemize}
\item closures of algebraic tangles (in Conway's sense)~\cite{PrzytyckiThreTtalksCuautitlanGeneralTitle};
\item knots with crossing number at most $12$~\cite{DabkowskiJabanKhanSahi4MoveEqKnotLinkTwoComp};
\item closures of $3$-strand braids~\cite{PrzytyckiSlavikJablanwork4moves}.
\end{itemize}
For a general overview of the subject, we refer to the survey~\cite{PrzytyckiSlavikJablanwork4moves}.

In this paper, we study an analogue of \textit{4-}equivalence in the setting of \emph{welded knot theory}, introduced in~\cite{FennRimanyiRourkeBraidPermuGrp}. Welded knot theory is a quotient of \emph{virtual knot theory}, a diagrammatic extension of classical knot theory allowing virtual crossings in addition to classical ones.
 
We first introduce a new local move, called the \textit{wink-move}, and the associated equivalence relation, called \textit{wink-equivalence}, generated by wink-moves together with welded isotopy. Although the wink-move is not a direct geometric analogue of the classical \textit{4-}move, it still provides a natural welded extension of classical \textit{4-}equivalence. Indeed, \textit{4-}equivalence implies wink-equivalence (Proposition~\ref{prop4implywink}), as the \textit{4-}move can be realized by wink-moves and welded isotopy. Moreover, the wink-move induces the same relations on quandle colorings as the classical \textit{4-}move. This motivates the notion of a \emph{\textit{4-}coherent quandle}, which provides a common algebraic framework for both the classical and welded settings. In particular, the invariants introduced in~\cite{DabkowskiPrzytyckiUnexpConnBetweenBurnsideGrpsKnotTheory,DabkowskiSahiInv4moves,BrittenhamHermillerTodd4movesDaboSahiInvKnot} remain invariant under wink-equivalence.

Inspired by Nakanishi's conjecture, we ask the following question.

\begin{ques}
Is every welded knot wink-equivalent to the unknot?
\end{ques}

As in the classical setting, our main result gives a positive answer for knots obtained as closures of \emph{algebraic virtual tangles}.

\begin{thm0}[Corollary~\ref{coroalgknotwinktriv}]
Every knot obtained as the closure of an algebraic virtual tangle is wink-equivalent to the unknot.
\end{thm0}

The proof relies on a version of \emph{arrow calculus}~\cite{JBYasuArrowcalc}, the welded counterpart of \emph{clasper calculus}~\cite{HabiroClasp}, adapted to wink-equivalence. We establish several basic relations in this setting and use them to obtain a finite classification of algebraic virtual tangles without closed components up to wink-equivalence. The theorem above follows from this classification.

\bigskip

The paper is organized as follows. Section~\ref{sectionwink} introduces wink-equivalence, discusses its relation with quandle colorings, and develops the corresponding arrow calculus. In Section~\ref{sectiontangle}, we prove the main theorem by classifying algebraic virtual tangles without closed components up to wink-equivalence.

\medskip\noindent\emph{Acknowledgements.}
The author is supported by JSPS KAKENHI Grant Number JP24KF0200. The author thanks Kazuo Habiro, Jean--Baptiste Meilhan and Akira Yasuhara for their careful reading and helpful comments.

\section{Wink-equivalence}\label{sectionwink}

\subsection{Wink-equivalence}

We start with the basic definitions.

\begin{defn}\label{defdiagram}
A \emph{\textbf{virtual diagram}} is the image of a proper smooth immersion of a one-dimensional manifold (a disjoint union of circles and intervals) into a disk, where each transverse double point is labelled as either a classical or a virtual crossing, as shown in Figure~\ref{figcrossings}. Here, \emph{proper} means that the boundary of the $1$-manifold is mapped to the boundary of the disk.

\begin{figure}[!htbp]
    \centering
    \includegraphics{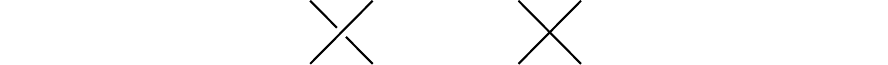}
    \caption{A classical and a virtual crossing.}
    \label{figcrossings}
\end{figure}
\end{defn}

In what follows, we will often simply say \lq diagram\rq\, instead of virtual diagram. A diagram without virtual crossings is called \emph{classical}.

\begin{defn}\label{defwink}
The equivalence relation of \emph{\textbf{welded isotopy}} on diagrams is generated by ambient isotopies fixing the boundary pointwise and the following local moves:
\begin{itemize}
    \item classical Reidemeister moves,
    \[\includegraphics{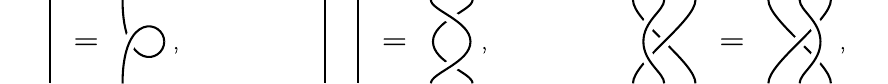}\]

    \item virtual Reidemeister moves,
    \[\includegraphics{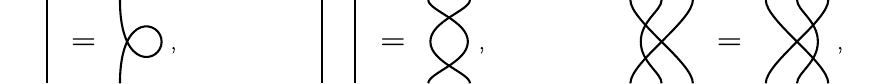}\]

    \item mixed Reidemeister move, $\hspace{4cm}\bullet$ (OC)-move,
    \[\includegraphics{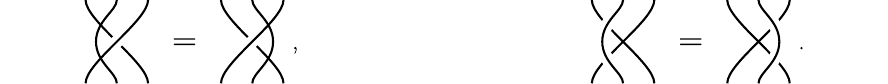}\]
\end{itemize}

The equivalence relation of \emph{\textbf{wink-equivalence}}  on diagrams is generated by welded isotopy and one of the following moves:
\begin{itemize}
    \item wink-moves,
    \[\includegraphics{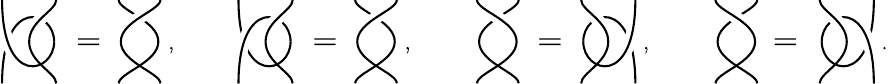}\]
\end{itemize}

Any version of the wink-move generates the same equivalence relation. The first two versions are related by adding a classical crossing to the left of the move, while the last two are obtained from the first two by adding a virtual crossing to the right.
\end{defn}

Although the wink-move does not look like the classical \textit{4-}move, it is a natural welded analogue for two reasons. First, Proposition~\ref{prop4implywink} shows that the \textit{4-}move can be realized by wink-equivalence, so wink-equivalence extends classical \textit{4-}equivalence. Second, as we will see in the next subsection, the wink-move induces the same relations on quandle colorings as the classical \textit{4-}move. The name is inspired by the appearance of the move: its two sides resemble a pair of eyes, one open and one closed.

\begin{prop}\label{prop4implywink}
\textit{4-}equivalence implies wink-equivalence.
\end{prop}

\begin{proof}
It suffices to show that wink-equivalence realizes the \textit{4-}move, as illustrated in Figure~\ref{figwelded4imply4}.

\begin{figure}[!htbp]
\centering
\includegraphics{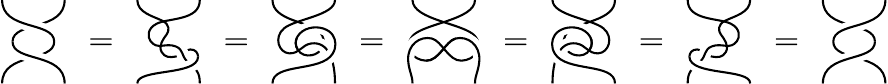}
\caption{A wink-equivalence realizing a \textit{4-}move.}
\label{figwelded4imply4}
\end{figure}

The first equality is a virtual Reidemeister~II move. The second is a wink-move, and the third follows from a welded isotopy containing an (OC)-move. The remaining equalities are deduced from the first three by symmetry
\end{proof}

Note that the proof of Proposition~\ref{prop4implywink} relies on an (OC)-move, which is forbidden in the virtual setting. Consequently, the wink-move does not lead to a virtual analogue of \textit{4-}equivalence but is instead specific to the welded setting.

\subsection{Quandle and \textit{4-}coherence}

We now study the interaction between \textit{4-}equivalence, wink-equivalence, and \emph{quandle colorings}. This naturally leads to the notion a \emph{\textit{4-}coherent} quandle.
\begin{defn}\label{defquandle}
A \emph{\textbf{quandle}} is a nonempty set \(Q\) equipped with two binary operations
\[
\rhd,\ \rhd^{-1} : Q \times Q \to Q
\]
such that, for all \(x,y,z\in Q\), the following hold:
\begin{enumerate}
    \item \(x \rhd x = x\);
    
    \item \((x \rhd y)\rhd^{-1} y = x = (x \rhd^{-1} y)\rhd y\);
    
    \item \((x \rhd y)\rhd z = (x \rhd z)\rhd (y \rhd z)\).
\end{enumerate}

\noindent A quandle is said to be \(\mathbf{4}\)\textbf{-coherent} if, for all \(x,y\in Q\), it satisfies the additional relations
\begin{enumerate}
    \item[(4)] \((x \rhd y)\rhd x = (x \rhd y)\rhd^{-1} x = x \rhd^{-1} y\).
\end{enumerate}
\end{defn}

A basic example is the following.

\begin{ex}\label{exquandleFQ2}
The algebra
\[
\mathbf{Z}_2[a]\big/(a+1)^3,
\]
equipped with the binary operations \(\rhd\) and \(\rhd^{-1}\) defined by
\[
x \rhd y = ax + (1+a)y,
\]
and
\[
x \rhd^{-1} y
   = a^{-1}x + (1+a^{-1})y
   = (1+a+a^2)x + (a+a^2)y,
\]
is a \(4\)-coherent quandle. In fact, one may verify that this quandle is the free \(4\)-coherent quandle on two generators.
\end{ex}

Recall the notion of a \emph{quandle coloring} of a diagram, where an \emph{arc} is a connected component of the diagram minus its crossings.

\begin{defn}\label{defcolor}
A \emph{\textbf{coloring}} of a diagram by a quandle is an assignment of a quandle element to each arc such that, for a chosen orientation of the diagram, the colors at every crossing satisfy the rules shown in Figure~\ref{Figcoloredcrossing}.

\begin{figure}[!htbp]
\centering
\includegraphics{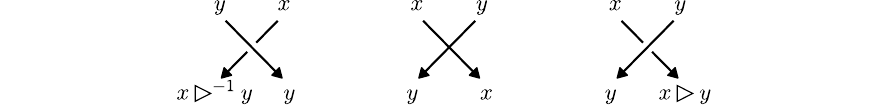}
\caption{Quandle coloring rules at crossings.}
\label{Figcoloredcrossing}
\end{figure}
\end{defn}

If two diagrams are related by one of the welded isotopy local moves from Definition~\ref{defwink}, then any coloring of one diagram induces a unique coloring of the other, without changing the colors of the arcs outside the region where the move is performed. The quandle axioms are exactly the relations needed for colorings to be compatible with the local moves generating welded isotopy. This is standard in both classical and welded knot theory; see, for instance, \cite{JoyceClassiInvKnotsQuandle} and \cite[\S8]{LiLeiWuUnknottingNumberWeldedKnot}. 

We now determine the additional relations imposed by a \textit{4-}move.

\begin{figure}[!htbp]
\centering
\includegraphics{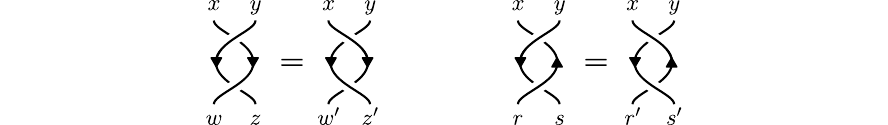}
\caption{Colorings of a \textit{4-}move.}
\label{Fig4movecolor}
\end{figure}

With the notation of Figure~\ref{Fig4movecolor}, an application of the coloring rules gives
\[
\left\{
\begin{aligned}
w  &= x \rhd y,\\
w' &= (x \rhd^{-1} y)\rhd^{-1} x,
\end{aligned}
\right.
\qquad
\left\{
\begin{aligned}
z  &= (y \rhd x)\rhd y,\\
z' &= y \rhd^{-1} x,
\end{aligned}
\right.
\qquad
\left\{
\begin{aligned}
r  &= x \rhd y,\\
r' &= (x \rhd^{-1} y)\rhd x,
\end{aligned}
\right.
\qquad
\left\{
\begin{aligned}
s  &= (y \rhd^{-1} x)\rhd y,\\
s' &= y \rhd x.
\end{aligned}
\right.
\]
Hence, a coloring extends across the \textit{4-}move if and only if the defining relations of a \textit{4-}coherent quandle are satisfied. Depending on the orientation, this amounts to either
\[
w=w'
\quad\text{and}\quad
z=z',
\]
or
\[
r=r'
\quad\text{and}\quad
s=s'.
\]

An analogous computation applies to the wink-move, shown in Figure~\ref{Figwelded4movecolor}.
\begin{figure}[!htbp]
\centering
\includegraphics{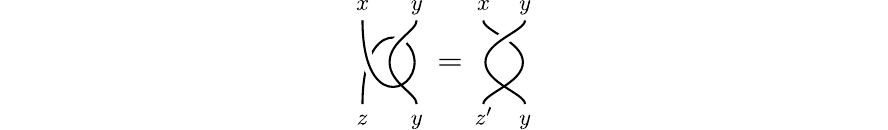}
\caption{Colorings of a wink-move.}
\label{Figwelded4movecolor}
\end{figure}

Depending on the orientation of the strands, one of the following identities holds:
\[
z=(x \rhd^{-1}y)\rhd^{\pm1}x,
\qquad
z'=x \rhd y,
\]
or
\[
z=(x \rhd y)\rhd^{\pm1}x,
\qquad
z'=x \rhd^{-1}y.
\]
Thus, a coloring is preserved by the wink-move precisely when $z = z'$, which is again equivalent to the defining relations of a \textit{4-}coherent quandle. Consequently, the wink-move induces the same relations on quandle colorings as the classical \textit{4-}move, making it a natural welded analogue from the viewpoint of quandle theory.

\subsection{Arrow calculus and wink-equivalence}
The main tool used in this paper is the \emph{arrow calculus} introduced in~\cite{JBYasuArrowcalc}. We adapt it here to wink-equivalence.

\begin{defn}\cite[Def.~3.1]{JBYasuArrowcalc}\label{defwarrow}
A collection of \emph{$\textbf{w}$\textbf{-arrows}} on a non-empty virtual diagram is a finite family of images of smoothly immersed intervals satisfying the following conditions:
\begin{itemize}
    \item Each $w$-arrow is oriented from one endpoint, called the \emph{\textbf{tail}}, to the other, called the \emph{\textbf{head}}. These endpoints are pairwise distinct and lie in the interior of the arcs of the diagram.
    \item The remaining singularities involving $w$-arrows consist of finitely many transverse double points, each labelled as a virtual crossing.
    \item Each $w$-arrow carries a finite number (possibly zero) of decorations, called \emph{\textbf{twists}}, disjoint from all singularities. By convention, two twists on the same edge cancel each other.
\end{itemize}
\end{defn}

In the figures, portions of the diagram are drawn in black and $w$-arrows in red. Heads are represented by arrows~\includegraphics{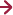}, and twists by red dots~\includegraphics{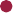}. We also use red dots with white interiors~\includegraphics{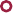} to indicate that a twist \textit{may or may} not be present on a given edge.

Given a diagram equipped with $w$-arrows, one may perform, surgery defined in~\cite{JBYasuArrowcalc}, to obtain a new diagram. Since we work with unoriented diagrams, we use the convention\textsuperscript{\hyperref[fn:surgery-convention]{\ref*{fn:surgery-convention}}} shown in Figure~\ref{arrowsurgery}, which differs slightly from that of~\cite{JBYasuArrowcalc}. Surgery on a collection of $w$-arrows is obtained by applying this operation to each $w$-arrow simultaneously.

\begin{figure}[!htbp]
\centering
\includegraphics{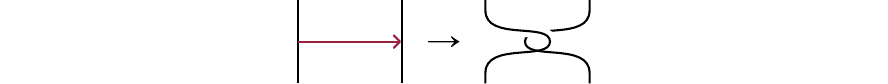}
\caption{Surgery on a $w$-arrow.}
\label{arrowsurgery}
\end{figure}

If a $w$-arrow contains twists, surgery introduces virtual crossings according to the convention shown on the left-hand side of Figure~\ref{arrowsurgery2}. Similarly, intersections between a $w$-arrow and either the diagram or another $w$-arrow produce virtual crossings according to the conventions shown in the middle and right-hand sides of Figure~\ref{arrowsurgery2}.
\begin{figure}[!htbp]
    \centering
    \includegraphics{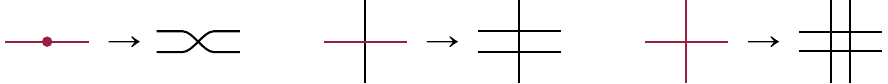}
    \caption{Surgeries near a twist and crossings.}
    \label{arrowsurgery2}
\end{figure}

The remainder of this section develops an \emph{\textbf{arrow calculus}} up to wink-equivalence, namely a collection of operations on diagrams equipped with $w$-arrows that preserve the wink-equivalence class of the surgery.

We first recall the \emph{\textbf{arrow moves}} from~\cite[\S4.3]{JBYasuArrowcalc}, which encode welded isotopy.

\begin{itemize}
    \item Arrow isotopy, consisting of virtual Reidemeister moves involving edges of $w$-arrows and/or strands of the diagram, together with the following local moves\footnote{Each vertical bicolored strand may represent either a portion of a $w$-arrow or of the diagram.},
    \[\includegraphics{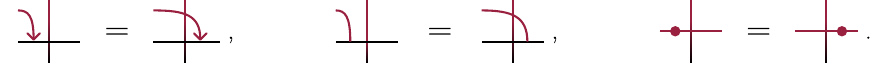}\]

    \item Head/Tail Reversal\footnote{\label{fn:surgery-convention}Because we use a different surgery convention, adapted to unoriented diagrams, the tail reversal move carries one additional twist compared with the version in~\cite{JBYasuArrowcalc}.}
    \[\includegraphics{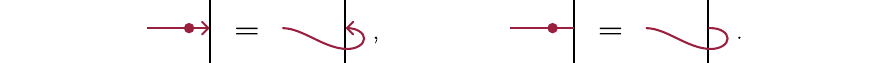}\]

    \item Tails Exchange,
    \[\includegraphics{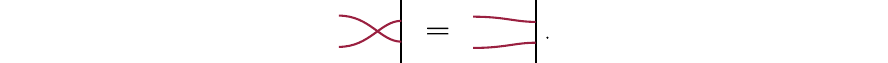}\]

    \item Isolated arrow,
    \[\includegraphics{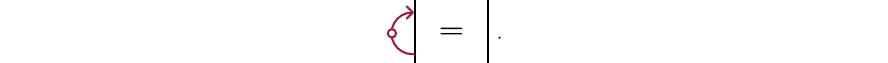}\]

    \item Arrow inverse,
    \[\includegraphics{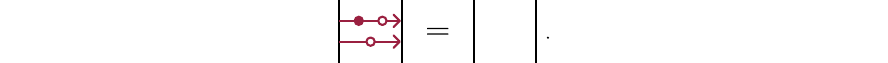}\]

    \item Slide,
    \[\includegraphics{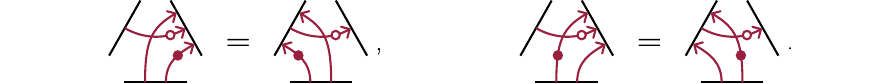}\]
\end{itemize}

The next move is the arrow calculus counterpart of the wink-move.

\begin{defn}\label{defarrow4move}
The local move shown in Figure~\ref{figwelded4arrowmove}, which removes a $w$-arrow enclosing the head of another $w$-arrow while adding a twist to the latter, is called the \emph{\textbf{Arrow~wink-move}}.

\begin{figure}[!htbp]
\centering
\includegraphics{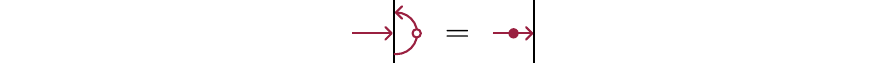}
\caption{Arrow~wink-move.}
\label{figwelded4arrowmove}
\end{figure}
\end{defn}
\begin{rmq}
Applying surgery to both sides of Figure 9 yields diagrams related by a single wink-move, up to welded isotopy. Thus, the Arrow wink-move realizes the wink-move within arrow calculus.
\end{rmq}

The following lemmas are used in the next section to handle $w$-arrows up to wink-equivalence.
\begin{lem}\label{lemarroworder4}
Four parallel $w$-arrows with the same orientation and number of twists cancel each other up to wink-equivalence.
\end{lem}

\begin{proof}
By Arrow inverse, it is enough to show that two parallel $w$-arrows are wink-equivalent to the same pair of $w$-arrows with one additional twist on each, as shown in Figure~\ref{figlemarroworder4}.

\begin{figure}[!htbp]
\centering
\includegraphics{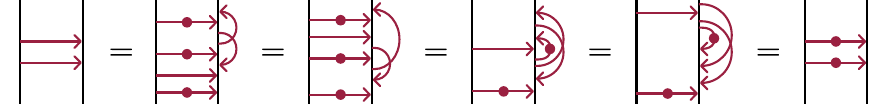}
\caption{Proof of Lemma~\ref{lemarroworder4}.}
\label{figlemarroworder4}
\end{figure}

The first equality uses two Arrow~wink-moves, one Tails Exchange and one Arrow inverse. The second is obtained by sliding the two middle $w$-arrows along the descending $w$-arrow on the right. The third follows from Arrow inverse and an Arrow~wink-move. The fourth is another Slide move, sliding the two right ascending $w$-arrows along the descending one. The last equality follows from an Isolated arrow together with two Arrow~wink-moves.
\end{proof}

\begin{lem}\label{lemarrowcommu}
Two parallel $w$-arrows with opposite orientations can be exchanged, up to wink-equivalence, by adding one twist to each of them.
\end{lem}
\quad
\begin{proof}
Figure~\ref{figlemarrowcommu} treats the case where neither $w$-arrow carries a twist. The general case follows from a similar sequence of equalities.
\begin{figure}[!htbp]
\centering
\includegraphics{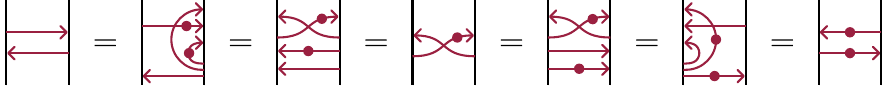}
\caption{Proof of Lemma~\ref{lemarrowcommu}.}
\label{figlemarrowcommu}
\end{figure}

The first equality follows from an Arrow~wink-move and an Isolated arrow. The second is given by a Slide move and a Tails Exchange, and the third an Arrow inverse. The last three equalities are the same as the first three, applied in reverse order. 
\end{proof}

\begin{lem}\label{lemarrowsquareenclosed}
If the heads of two parallel $w$-arrows having the same orientation are enclosed by another $w$-arrow, then the enclosing $w$-arrow can be removed up to wink-equivalence.
\end{lem}

\begin{proof}
We only consider the case where the enclosing $w$-arrow is ascending (Figure~\ref{figlemarrowsquareenclosed}); the descending case is symmetric.

\begin{figure}[!htbp]
\centering
\includegraphics{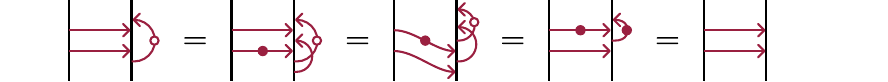}
\caption{Proof of Lemma~\ref{lemarrowsquareenclosed}.}
\label{figlemarrowsquareenclosed}
\end{figure}

The first equality follows from an Arrow~wink-move and a Tails Exchange. The second is obtained by a Slide move. The remaining equalities follow from two further Arrow~wink-moves.
\end{proof}

\section{Virtual Tangles}\label{sectiontangle}

\begin{defn}\label{defvirtualtangle}
A \emph{\textbf{virtual tangle}} is a virtual diagram whose ambient disk is identified with the unit square $[0,1]^2$. Its underlying one-dimensional manifold consists of two intervals together with finitely many circles. The endpoints of the intervals are required to lie at the four corners of the square, namely the \textit{top-left}, \textit{top-right}, \textit{bottom-left}, and \textit{bottom-right} corners.
\end{defn}

There are two natural composition laws on virtual tangles: the \emph{\textbf{vertical stacking operation}}, denoted by $\vstack$, and the \emph{\textbf{horizontal stacking operation}}, denoted by $\hstack$, illustrated in Figure~\ref{figtanglestack}.

\begin{figure}[!htbp]
\centering
\includegraphics{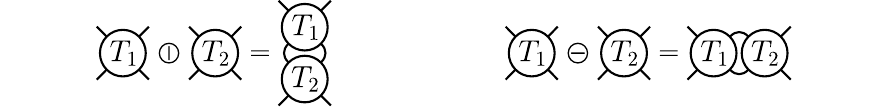}
\caption{Vertical and horizontal stacking operations on virtual tangles.}
\label{figtanglestack}
\end{figure}

We also consider the \emph{\textbf{vertical closure}} and \emph{\textbf{horizontal closure}} operations, which produce \textit{virtual link diagrams} by connecting the boundary points as shown in Figure~\ref{figtangleclosure}.

\begin{figure}[!htbp]
    \centering
    \includegraphics{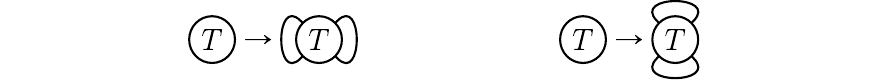}
    \caption{Vertical and horizontal closures of virtual tangles.}
    \label{figtangleclosure}
\end{figure}

\begin{defn}\label{defalgvirtualtangle}
An \emph{\textbf{algebraic virtual tangle}} is any virtual tangle generated from the following elementary tangles using the stacking operations:
\[
\includegraphics{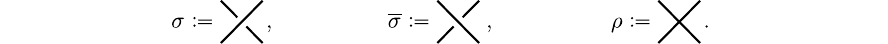}
\]

We similarly define the sets of \emph{\textbf{vertical algebraic virtual tangles}} and \emph{\textbf{horizontal algebraic virtual tangles}} by allowing only one of the two stacking operations.
\end{defn}

Let $\mathcal{T}$ denote the set of algebraic virtual tangles without closed components up to wink-equivalence. Likewise, let $\mathcal{T}_{\subvstack}$ and $\mathcal{T}_{\subhstack}$ denote the sets of vertical and horizontal algebraic virtual tangles up to wink-equivalence.

The classical and virtual Reidemeister~II moves endow $\mathcal{T}_{\subvstack}$ and $\mathcal{T}_{\subhstack}$ with natural group structures. Indeed, they can be identified with quotients of the \emph{two-strand virtual braid group} by wink-equivalence. Rotation by $90^\circ$ also induces a natural group isomorphism between $\mathcal{T}_{\subvstack}$ and $\mathcal{T}_{\subhstack}$.

To state the next result, we introduce the following tangles,

\begin{figure}[!htbp]
\centering
\includegraphics{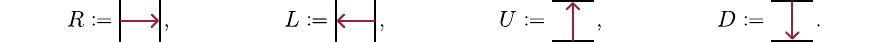}
\end{figure}
\noindent Their expressions are given by
\[R=\overline{\sigma}\vstack\rho,\hspace{53.45pt} L=\rho\vstack\overline{\sigma},\hspace{53.45pt} U=\sigma\hstack\rho,\hspace{53.45pt} D=\rho\hstack\sigma.\]

\begin{prop}\label{proptanglesets}
The sets $\mathcal{T}_{\subvstack}$ and $\mathcal{T}_{\subhstack}$ are finite and are given by
\[
\mathcal{T}_{\subvstack}
=
\{R^u \vstack L^v \vstack \rho^w \mid u,v\in\{0,1,2,3\},\; w\in\{0,1\}\},
\]
and
\[
\mathcal{T}_{\subhstack}
=
\{U^u \hstack D^v \hstack \rho^w \mid u,v\in\{0,1,2,3\},\; w\in\{0,1\}\}.
\]
The choice $u=v=w=0$ corresponds to the vertical $\infty$-tangle and the horizontal $0$-tangle, respectively, consisting of two parallel vertical strands and two parallel horizontal strands.
\end{prop}

\begin{proof}
Since rotation identifies the two sets, it is enough to prove the result for $\mathcal{T}_{\subvstack}$.

We first show that $\mathcal{T}_{\subvstack}$ is closed under vertical stacking. Write
\[
T_i\coloneqq R^{u_i}\vstack L^{v_i}\vstack \rho^{w_i},
\qquad i=1,2,
\]
where $u_i,v_i\in{0,1,2,3}$ and $w_i\in{0,1}$. By virtual Reidemeister~II move,
\[
\rho\vstack R = L\vstack \rho,
\qquad
\rho\vstack L = R\vstack \rho.
\]
Hence
\[
T_1\vstack T_2=
\begin{cases}
R^{u_1}\vstack L^{v_1}\vstack R^{u_2}\vstack L^{v_2}\vstack \rho^{w_2},
& \text{if } w_1=0,\\[1ex]
R^{u_1}\vstack L^{v_1+u_2}\vstack R^{v_2}\vstack \rho^{1+w_2},
& \text{if } w_1=1.
\end{cases}
\]
Applying Lemma~\ref{lemarrowcommu} yields
\[
T_1\vstack T_2=
\begin{cases}
R^{u_1-u_2}\vstack L^{-v_1+v_2}\vstack \rho^{w_2},
& \text{if } w_1=0,\\[1ex]
R^{u_1-v_2}\vstack L^{-v_1-u_2}\vstack \rho^{1+w_2},
& \text{if } w_1=1.
\end{cases}
\]
The exponent of $\rho$ can be reduced modulo $2$ using the virtual Reidemeister~II move, while Lemma~\ref{lemarroworder4} allows the exponents of $R$ and $L$ to be reduced modulo $4$. Thus every product has the required form.

It remains to check that the elementary tangles belong to this set:
\[
\sigma = L^{-1}\vstack\rho,\qquad
\bar{\sigma}=R\vstack\rho,\qquad
\rho=\rho.
\]
This completes the proof. 
\end{proof}

\begin{rmq}
The elements of $\mathcal{T}_{\subvstack}$ listed in Proposition~\ref{proptanglesets} are pairwise distinct. Indeed, consider colorings by the \textit{4-}coherent quandle
\[
\mathbf{Z}_2[a]\big/(a+1)^3
\]
from Example~\ref{exquandleFQ2}. Assign the colors $0$ and $1$ to the \textit{top-left} and \textit{top-right} strands, respectively, and propagate them using the coloring rules of Figure~\ref{Figcoloredcrossing}. The resulting colors on the \textit{bottom-left} and \textit{bottom-right} strands determine the values of $u$, $v$, and $w$, and therefore distinguish all elements of $\mathcal{T}_{\subvstack}$. Since rotation identifies $\mathcal{T}_{\subvstack}$ and $\mathcal{T}_{\subhstack}$, the analogous statement holds for $\mathcal{T}_{\subhstack}$.
\end{rmq}

For Proposition~\ref{proptantglesetintersection} and Theorem~\ref{thmtangleset}, we introduce the following convention.

\begin{conv}
In the figures, a $w$-arrow labelled by a circled integer $u\in\mathbf Z$ denotes a stack of $|u|$ parallel $w$-arrowss, all with the same orientation. If $u<0$, each $w$-arrow in the stack carries an additional twist. This convention is illustrated in Figure~\ref{figarrowstackconvention}.

\begin{figure}[!htbp]
\centering
\includegraphics{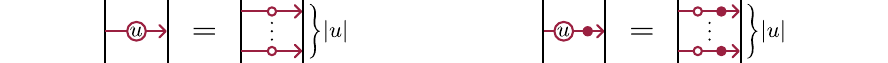}
\caption{Convention for the stacking of $u$ parallel $w$-arrows.}
\label{figarrowstackconvention}
\end{figure}
\end{conv}

\begin{prop}\label{proptantglesetintersection}
The intersection of $\mathcal{T}_{\subhstack}$ and $\mathcal{T}_{\subvstack}$ is
\[
\mathcal{T}_{\subhstack}\cap\mathcal{T}_{\subvstack}
=
\{T^u \hstack D^v \hstack \rho \mid u,v\in\{0,1,2,3\}\}
=
\{R^u \vstack L^v \vstack \rho \mid u,v\in\{0,1,2,3\}\}
.
\]
\end{prop}

\begin{proof}
By Proposition~\ref{proptanglesets}, any tangle in the intersection must connect the boundary points along the same diagonal, namely from the \textit{top-left} to the \textit{bottom-right} and from the \textit{top-right} to the \textit{bottom-left}. Hence, every element of the intersection belongs to the sets appearing in the statement.

It therefore remains to show that the two families given in the statement coincide. This follows from the chain of equalities shown in Figure~\ref{figproptantglesetintersection}.

\begin{figure}[!htbp]
\centering
\includegraphics{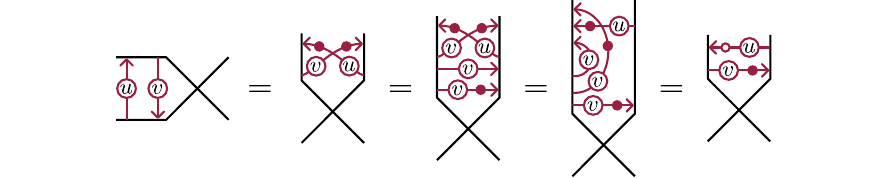}
\caption{Proof of Proposition~\ref{proptantglesetintersection}.}
\label{figproptantglesetintersection}
\end{figure}

The first equality follows from Arrow isotopy, Head/Tail Reversal and Tails Exchange, and already proves the result when $u=0$. The second follows from Arrow inverse, and the third from Slide move. When $u=2$, one first applies Arrow inverse to introduce $w$-arrows between the two tails of the \emph{$u$-factor}, and then performs the sliding. The last equality follows from Isolated arrow together with $v$~applications of the Arrow~wink-move when $u=\pm1$, and of Lemma~\ref{lemarrowsquareenclosed} when $u=2$. During this process, the \emph{$u$-factor} may acquire twists due to the Arrow~wink-moves. Since twists are involutive, the resulting twisting depends only on the parity of $v$. Finally, to recover the expression of the statement it remains to exchange the $u$-factor with the $v$-factor using Lemma~\ref{lemarrowcommu}.
\end{proof}

\begin{thm}\label{thmtangleset}
The set $\mathcal{T}$ is finite and is given by
\[
\mathcal{T}
=
\{X(u,v),\,H(u,v),\,V(u,v)
\mid u,v\in\{0,1,2,3\}\},
\]
where
\[
X(u,v)=T^u \hstack D^v \hstack \rho,
\qquad
H(u,v)=T^u \hstack D^v,
\qquad
V(u,v)=R^u \vstack L^v.
\]
\end{thm}

\begin{proof}
The elementary tangles $\sigma$, $\overline{\sigma}$ and $\rho$ all belong to the set above. It is therefore enough to prove that this set is closed under the stacking operations $\vstack$ and $\hstack$, whenever no closed component is created.

By Proposition~\ref{proptantglesetintersection} and Lemma~\ref{lemarrowcommu}, rotation fixes the tangles $X(u,v)$ and exchanges $V(u,v)$ with $H(u,v)$. Since $\mathcal T$ is invariant under rotation, and rotation exchanges vertical and horizontal stacking, it is enough to establish closure under one of the two operations.

We therefore consider the nine possible horizontal stackings. The resulting configurations fall into four distinct cases, denoted by $A$, $B$, $C$, and $D$. Table~\ref{tablestack} lists the case corresponding to each of the nine stackings.

\begin{table}[!htbp]
\centering
\begin{tabular}{|c|c|c|c|}
\hline
$\hstack$ & $X(u,v)$ & $H(u,v)$ & $V(u,v)$\\
\hline
$X(u,v)$ & $A$ & $A$ & $C$\\
\hline
$H(u,v)$ & $A$ & $A$ & $B$\\
\hline
$V(u,v)$ & $C$ & $B$ & $D$\\
\hline
\end{tabular}
\vspace{0.3cm}
\caption{The nine possible horizontal stackings.}
\label{tablestack}
\end{table}
\vspace{-0.4cm}
The $D$-cases cannot occur, since such a product always creates a closed component. The four $A$-cases are covered by Proposition~\ref{proptanglesets}, because both factors belong to $\mathcal T_{\subhstack}$. The two $B$-cases are equivalent by rotation and follow essentially from the equalities in Figure~\ref{figthmproofBcases}.

\begin{figure}[!htbp]
\centering
\includegraphics{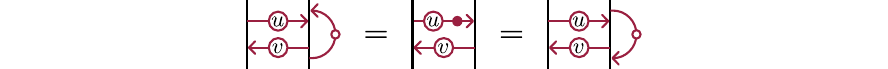}
\caption{Equalities used to prove the $B$-cases.}
\label{figthmproofBcases}
\end{figure}

The first equality follows from Tails Exchanges together with Isolated arrow when $u=0$, an arrow~wink-move when $u=\pm1$, and Lemma~\ref{lemarrowsquareenclosed} when $u=2$. The second is proved in the same way, with an additional application of Lemma~\ref{lemarrowcommu} to exchange the \emph{$u$-} and \emph{$v$-factors}.

The two $C$-cases are likewise equivalent by rotation. Their proof combines the equalities of Figure~\ref{figthmproofBcases} with those of Figure~\ref{figthmproofCcases}.

\begin{figure}[!htbp]
\centering
\includegraphics{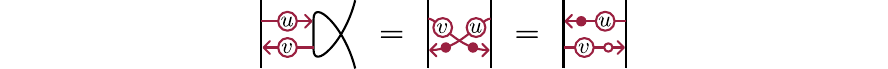}
\caption{Equalities used to prove the $C$-cases.}
\label{figthmproofCcases}
\end{figure}

The first equality follows from the Arrow isotopy, Head/Tail Reversal, and Tails Exchange moves. A similar version of the second equality appears in Figure~\ref{figproptantglesetintersection} and was discussed in the proof of Proposition~\ref{proptantglesetintersection}.
\end{proof}

\begin{cor}\label{coroalgknotwinktriv}
Knots obtained as closures of algebraic virtual tangles are trivial up to wink-equivalence.
\end{cor}

\begin{proof}
By Theorem~\ref{thmtangleset}, only finitely many tangles need to be considered, namely the closures of the tangles $X(u,v)$, $H(u,v)$ and $V(u,v)$ that are knots. By rotation, it is enough to consider the vertical closures of $X(u,v)$ and $H(u,v)$.

Each of these closures is wink-equivalent to the unknot by Isolated arrow and a Reidemeister~I move.
\end{proof}

\bibliographystyle{alpha}
\bibliography{bibli4eq.bib}
\addcontentsline{toc}{chapter}{Bibliography}

\end{document}